\documentclass[a4paper,11pt]{article}
\usepackage[cp1251]{inputenc}
\usepackage{amsfonts, amssymb, amsmath, amsthm}
\usepackage[english]{babel}
\newtheorem{thrm}{Theorem}[section]
\newtheorem{lem}[thrm]{Lemma}

\newtheorem{cor}[thrm]{Corollary}
\theoremstyle{definition}

\newtheorem{remark}[thrm]{Remark}
\numberwithin{equation}{section}

\newtheorem{atheorem}{Theorem}

\newtheorem{alemma}{Lemma}[atheorem]

\title{Estimates of conjugate harmonic functions with given set of singularities with application}

\author{I. Chyzhykov, Yu. Kosanyak}
\newcommand{\Bl}{\Bigl(}
\newcommand{\Br}{\Bigr)}

\def\arg{\mathop{\mbox{\rm arg}}}

\def\vfi{\varphi}

\def\ve{\varepsilon}

\def\al{\alpha}

\def\R{\mathbb R}
\def\C{\mathbb C}
\def\D{\mathbb D}
\def\Ga{\Gamma}
\def\ga{\gamma}

\def\ve{\varepsilon}

\def\tg{\mathop{\mbox{tg}}}

\renewcommand{\Re}{\mathop{\rm Re}}
\newcommand{\dist}{\mathop{\rm  dist}}
\renewcommand{\Im}{\mathop{\rm Im}}
\def\supp{\mathop{\mbox{supp}}}

\def\noi{\noindent}

\begin{document}

\maketitle

\begin{abstract} \normalsize

Let $E$ be an arbitrary  closed set on the unit circle $\partial \D$, u be a harmonic function on the unit disk $\D$ satisfying $|u(z)|\lesssim (1-|z|)^\gamma \rho^{-q}(z)$ where $\rho(z)= \mathop{\rm dist}(z, E)$, $\gamma$, $q$ are some real constants, $\gamma\le q$. We establish an estimate of the conjugate $\tilde u$ of the same type which is sharp in some sense and in the case $E=\partial D$ coincides with known estimates.  As an application we describe  growth classes  defined by the non-radial condition $|u(z)|\lesssim \rho^{-q}(z)$ in terms of smoothness of the Stieltjes measure associated to the harmonic function $u$.
%Answering a question of A.Zygmund,  G.MacLane and L.Rubel described conditions of boundedness of the $L_2$-norm w.r.t. the argument of $\log |B|$, where $B$ is a Blaschke product. We generalize their results in several directions. In particular, we describe growth and decrease  of $p$th logarithmic means, $p\in(1, \infty)$,  of bounded analytic functions  in the unit disc in terms of its zeros and the boundary measure. We also prove sharp upper estimates  of $p$th means of analytic and subharmonic functions of finite order in the unit disc.

keywords: {conjugate harmonic function,  analytic function, unit disk, modulus of continuity}

Subjclass: {Primary 31A05, Secondary 30J99, 30E20, 31A10, 31A20}

\end{abstract}

%\maketitle

%$E=\{ e^{i\theta} : \theta \in E^*\}$, $

\section{Introduction and main results}

\subsection{Classes of analytic functions defined by a non-radial growth condition}

Let $\D=\{z\in \mathbb{C}:|z|<1\} $ denote the unit disk, $\partial \D$ be its boundary.  In the classical theory of analytic and harmonic functions in the unit disc \cite{Tsuji}, \cite{Dur} dominates the approach where the unit circle is considered as `one collective singularity', and main characteristics of a function   such as the maximum modulus $M(r,f)$, integral  means $M_p(r,f)=\Bl \frac 1{2\pi} \int_0^{2\pi} |f(re^{i\theta})|^p\, d\theta \Br^\frac 1p $, $0<p<\infty$, convergence exponent $\inf \{\mu> 0: \sum_n (1-|a_n|)^{\mu+1} <\infty  \} $ of a sequence $(a_n)$ in $\D$,   etc. depend on $r$ or $|a_n|$, or in other words, on the distance $1-r$  to the unit circle only.
This approach unable to  use the structure of the set $E\subset \partial \D$ of singular points  of an analytic function $f$. It is worth to note that many years ago before Tsuji, V. Golubev proposed another approach, published in his work \cite{Gol61}\footnote{The work was originally published by parts in 1924--27 in Uchenyje zapiski Gosudarstvennogo Saratovskogo universiteta.}.

Given an arbitrary compact set $E\subset \mathbb{C}$ of singularities of an analytic function $f$, Golubev studied growth, decrease, zero distribution of $f$ using characteristics  in terms of the quantity  $\rho(z)=\rho_E(z)=\dist (z, E)$. In particular, he constructed a theory of canonical products, which are now called Golubev products.
As an example we cite one his result.

\begin{atheorem} [{\cite[\S 3]{Gol61}}]
If $(a_n)$ is a sequence in $\C$ such that $p=\inf \{ \alpha>0: \sum_{n} \rho^\alpha (a_n) <\infty\}$ is finite, then there exists an analytic   function $\varphi(z)$ vanishing exactly on $(a_n)$ such that for an arbitrary $\ve>0$
$$ |\varphi(z)|< e^{\rho^{-p-\varepsilon}(z)}$$ in some neighborhood of $E$.
\end{atheorem}

Golubev's results remained unknown for a long time, but his approach has recently obtained new breath  in a series of works by A. Borichev, L. Golinski, S. Kupin \cite{BGK09}--\cite{BorGol12}, S. Favorov and L. Golinski \cite{FG09}--\cite{FG19}. In particular, in \cite{FG09} the authors have found optimal Blaschke-type condition for zero set for the class of analytic functions in the unit disc satisfying
\begin{equation}\label{e:anal_growth_con}
  |f(z)|\le e^{D\rho_E^{-q}(z)}
\end{equation}
where $E=\overline{E}\subset \partial \D$, $D$, $q$ are nonnegative constants. It is appeared that the notion of a type of a set $E$ \cite{AK} characterizing `sparseness' plays an important role. In \cite{BGK09} the authors consider classes of analytic functions in $\D$ defined by the distances to two disjoint finite sets on the unit circle. Results of such type effectively apply to the study of complex perturbation of certain self-adjoint operators.

\subsection{Conjugate harmonic functions}
Let $u$ be a harmonic function in  $\D$. We denote  $M_\infty(r,u)=\sup \{ |u(z)|: |z|=r\}$, $0\le r<1$. Estimates  of the growth of the harmonic conjugate $\tilde u(z)$ have many applications in the function theory (see \cite{Dur}, \cite{HL}, \cite{Z}). In particular, for $1<p<\infty$, by M. Riesz theorem (\cite{Z}) $M_p(r,u)\le  \psi\Bl \frac 1{1-r}\Br$  implies $M_p(r,\tilde u)=O\Bl \psi\Bl \frac 1{1-r}\Br \Br$, where $\psi$ is a positive increasing function on $[1, \infty)$, $r\in [0,1)$.
The cases $p=1$ and $p=\infty$ are more delicate.

Let  $\varphi\colon \R_+ \to\R_+$  be a nondecreasing function.
Suppose that $\varphi(x)$ is almost increasing for $x\ge 1/2$, i.e.
there exist positive constants  $a$ and $c$ such that
\begin{equation}\label{e:u_con_shields}
    \frac{\varphi(x_2)}{x_2^a}\le c     \frac{\varphi(x_1)}{x_1^a}, \quad x_2>x_2\ge \frac 12.
\end{equation}
Let \begin{equation}\label{e:psi_conjugate}
    \tilde{\varphi}(x)=\int_{\frac 12}^x \frac {\varphi(t)}t dt, \quad x\ge \frac 12.
\end{equation}
\begin{atheorem}\label{t:shields_will} \sl
Suppose that $u$ is harmonic in $\D$, $\varphi$ is an increasing function on $\R_+$, and  (\ref{e:u_con_shields}) holds.
Then  \begin{itemize}
       \item [1)]  $M_\infty(r,u)=O\Bl \varphi \Bl \frac 1{1-r}  \Br \Br$ implies  $M_\infty(r,\tilde u)=O\Bl \tilde \varphi \Bl \frac 1{1-r}  \Br \Br$, $r\to 1-$;
       \item [2)]  $M_1(r,u)=O\Bl \varphi \Bl \frac 1{1-r}  \Br \Br$ implies $M_1(r,\tilde u)=O\Bl \tilde \varphi \Bl \frac 1{1-r}  \Br \Br$, $r\to1-$.
     \end{itemize}
\end{atheorem}

Theorem \ref{t:shields_will} does not give us particular information  on local growth of $u$ and $\tilde u$ in a neighborhood of a point on $\partial \D$.

{\bf Example 1.} Let $f(z)=\Bl \frac{1+z}{1-z}\Br^{\frac 12}$, $z\in \D$ where the branch of the square root is chosen such that $f(0)=1$. We have $$ f(z)=\Bl \frac{1-|z|^2}{|1-z|^2}+ \frac{2i \Im z}{|1-z|^2}\Br^{\frac 12}=\Bigr| \frac{1+z}{1-z}\Bigr|^{\frac12}(\cos \frac{\gamma}{2}+ i\sin \frac{\gamma}{2})$$
where $\gamma =\arg \frac{1+z}{1-z}\in (-\frac \pi 2, \frac \pi 2)$, $\tg \gamma =\frac{2\Im z}{1-|z|^2}$, $z\in \D$.

It is clear that both $\Re f$ and $\Im f$ are bounded outside a neighborhood of $1$. Moreover,  $\cos \frac{\gamma}{2}\asymp 1$, because $\frac{\gamma}{2}\in (-\frac{\pi}{4},\frac{\pi}{4})$. Using  elementary trigonometry we get as $|z| \to 1-$, $z\in \D$
\begin{gather*}
  \sin^2 \frac \gamma 2 =\frac12 \Bl 1- \frac{1}{\sqrt{1+\tg ^2 \gamma}}\Br= \\
  =\frac 12 \Bl 1 - \frac{1-|z|^2}{\sqrt{(1-|z|^2)^2+ 4 (\Im z)^2}}\Br  \asymp \frac{|\Im z|^2}{|1-z|^2}.
\end{gather*}
Therefore
$$ \Re f(z)\asymp \frac{1}{|1-z|^{\frac 12}}, \quad \Im f(z)\asymp \frac{|\Im (1-z)|}{|1-z|^{\frac 32}}. $$
We see that conjugate harmonic functions have the same majorant   $(\rho(z))^{-\frac 12}$, $E=\{1\}$.

The problem of estimating the harmonic conjugates for harmonic functions from classes
\begin{equation}\label{e:harm_growth_con}
  |u(z)|\le \frac{D}{\rho_E^{q}(z)}, \quad z\in \D
\end{equation}
arises naturally.

In the sequel, the symbol $C$  with indices stands for some positive constants not necessary the same at each occurrence. The relation $A\lesssim B$ means that there exists a constant $C$ such that $A\le CB$. 

\begin{thrm}
    \label{t:harm_conj}
Let $u$ be a harmonic function in the unit disk, $\tilde u$ be its conjugate function. Let $E=\overline{E} \subset \partial \D$,  $q$,  $\gamma$ be real constants satisfying $q>0$, $\gamma\le q$. If for some $C_0>0$
\begin{equation}\label{e:u_est}
  |u(z)|\le \frac{C_0(1-|z|)^\gamma}{\rho^q(z)},
\end{equation}
then there exists $C_1=C_1(C_0, \lambda, q, \gamma)>0$ such that
\begin{equation}\label{e:v_est}
  |\tilde u(z)|\le C_1 \left\{
                  \begin{array}{ll} \displaystyle
                    \frac{\log \frac{1}{1-|z|}}{\rho^{q-\gamma}(z)}, & q>\gamma\ge 0; \\
                    \frac{(1-|z|)^\gamma}{\rho^{q}(z)}, & q>0>\gamma; \\
                    \log \frac{1}{\rho(z)}, & q=\gamma>0,
                  \end{array}
                \right. \quad \frac 12 \le  |z|<1.
\end{equation}
\end{thrm}

\begin{remark} If $E=\partial \D$, that is $\rho(z)=1-|z|$, the hyposthesis  $|u(z)|\lesssim (1-|z|)^{-p}$, $p>0$, by Theorem \ref{t:harm_conj} with $q=p/2$, $\gamma=-p/2$ implies $|\tilde u(z)|\lesssim (1-|z|)^{-p}$, and $|u(z)|\lesssim 1$ implies $|\tilde u(z)|\lesssim \log \frac 1{1-|z|}
$, which is known to be sharp  bounds.

\end{remark}
{\bf Example 2.}
Consider the Poisson kernel $P_0(z)$, which is the real part of the Schwarz kernel. Then $E=\{1\}$, and the assumptions of the theorem hold with $q=2$ and $\gamma =1$. It is clear that for the conjugate function $Q(z)=\frac{2 \Im(z-1)}{|1-z|^2}$ we have $|Q(z)|\asymp \frac1{|1-z|}$ as $1-z=e^{i\frac \pi 4t}$, $t\to 0+$. It means that the estimate given by \eqref{e:v_est} in the case $q>\gamma \ge 0$ is sharp up to the factor $\log \frac{1}{1-r}$.

\begin{remark} We do not know whether the factor $\log \frac 1{1-r}$ in \eqref{e:v_est} is necessary.
\end{remark}

\subsection{An application}
We give an application to the growth of Poisson-type integrals. In order to proceed we need the definition of the  Riemann-Liouville fractional integral.
%\section{An application}
The fractional integral $D^{-\alpha}h$ and the fractional derivative $D^{\alpha}h$ of order $\al>0$ for $h \colon (0,1) \to
\R$ are  defined by the formulas \cite{Dj} $$D^{-\al} h(r)=\frac
1{\Ga(\al)}\int _0^r (r-x)^{\al-1} h(x)\, dx,\quad 0<\alpha<1, $$ $$D^0h(r)\equiv
h(r), \quad D^\al h(r)=\frac{d^p}{dr^p} \{ D^{-(p-\al)} h(r)\}, \;
\al\in (p-1;p].$$

%Let $H(\D)$ be the class of harmonic functions in $\D$. 
We put
$u_\al(re^{i\vfi})=r^{-\al} D^{-\al} u(re^{i\vfi})$, where the
fractional integral is taken with respect to the variable $r$.
%We define $B(r,u)=\max \{
%|u(z)|: |z|\le r\}$ for a subharmonic function $u$ in $D$.
The reason of this definition is that $u_\alpha$ is harmonic in $\D$ (\cite[Chap. IX]{Dj}).

Let $$S_{\al}(z)=\Gamma(1+\al)\Bigl(
 \frac 2{(1-z)^{\al+1}} -1\Bigr), \quad P_\al(z)
=\Re S_\al(z).$$

\noi{\bf Remark 2.} Note that $S_0(z)$ is the Schwartz kernel,  $P_0(z)$ is the Poisson kernel;
$P_{\al}(re^{it})=D^\al(r^\al P_0(re^{it}))$.

Our starting  point is the following  two theorems. The first one is a representation theorem.

\medskip\noi
\begin{atheorem}[M.\ Djrbashian] \label{t:dj2} \sl Let $u$ be harmonic in $\D$, $\al
>-1$. Then
\begin{equation} \label{e:ual}
u(re^{i\vfi})=\int_0^{2\pi} P_\al(re^{i(\vfi-\theta)})\,
d\psi(\theta),
\end{equation}
where $\psi\in BV[0,2\pi]$,
 if and only if
$$ \sup_{0<r<1} \int_0^{2\pi} |u_\al(re^{i\vfi})|\, d\vfi<
+\infty.$$
\end{atheorem}

The idea of the next result goes back to Hardy and Littlewood \cite{HL} and consists in the observation that there is an iterplay between  the growth of Poisson-Stieltjes integral and the smoothness of the Stieltjes measure. Growth of $M_\infty(r,u)$ where $u$ is of the form \eqref{e:ual} was described in \cite{Ch1}, \cite{ChUMZ}. We define the modulus of continuity of a bounded function $\psi\colon [0, 2\pi]\to \mathbb{R}$ on a set $E\subset [0, 2\pi]$  by $$\omega_E(\delta, \psi)= \sup\{ |\psi(y)-\psi(x)| : |x-y|<\delta, x,y\in E\} ,$$ and  we write $\omega(\delta, \psi)=
\omega_{[0,2\pi]}(\delta, \psi)$ for short.

\begin{atheorem} \label{t:2}  {\sl Let $u$ be harmonic in $\D$, $\al \ge
0$, $0<\ga< 1$ or $\gamma=1$ and $\alpha=0$. Then $u(z)$ has form (\ref{e:ual})
%\begin{equation}
%u(re^{i\vfi})=\frac 1{2\pi} \int_0^{2\pi} P_\alpha(r, \vfi-t) d\psi(t)
%\end{equation}
where   $\psi $ is of bounded variation on $[0,2\pi]$, and $\omega(\delta, \psi)=O(\delta^\gamma)$, $\delta\to 0+$,
if and only if $$ M_\infty(r, u )=O((1-r)^{\gamma -\al
-1}), \quad r\to 1- \;$$ and $$ \sup\limits _{0<r<1} \int
_{0}^{2\pi} |u_\alpha(re^{i\vfi})|\,d\vfi <+\infty.$$ }
\end{atheorem}

Asymptotic behavior of $M_p(r,u)$ where $1<p<\infty$ is described in \cite{Chmm}, see also \cite{ChIJM}.
The disadvantage of the mentioned results is that they do not describe local growth of a harmonic function.

\medskip \noi
\begin{thrm}
\label{t:appl}
\sl Let  $0<\ga< 1 $, $\alpha\ge 0$,  and $u(z)$ have the form
\begin{equation} \label{e:rep1}
u(re^{i\vfi})=\int_0^{2\pi} P_\alpha(re^{i\vfi} \bar \zeta) d\mu(\zeta)
\end{equation}
where  $\mu$ is a real finite Borel measure on $\partial \D$,
%  $\psi $ is of bounded variation on $[0,2\pi]$,
 $\supp \mu=E\subset \partial \D$.
\begin{itemize}
  \item [(i)] If $\omega_E(\delta, \mu)=O(\delta^\gamma)$, $\delta \to 0+$, then
  \begin{equation}\label{e:u_growth}
    |u(z)|\le \frac{C}{\rho^{\alpha+ 1-\gamma}(z)};  \end{equation}
  \item [(ii)] If \eqref{e:u_growth} holds, then
   $$\omega_E(\delta, \mu)=O\Bl \delta^\gamma \log \frac{1}{\delta}\Br, \quad \delta\to 0+.$$
\end{itemize}
% and $\psi
%\in\La_\gamma$, if and only if $$ B(r, u )=O((1-r)^{\gamma
%-1}), \quad r\uparrow 1 \;$$ and $$ \sup\limits _{0<r<1} \int
%_{0}^{2\pi} |u(re^{i\vfi})|\,d\vfi <+\infty.$$ }
\end{thrm}
Given a set $E\subset \partial \D$ we define the order of growth $$\sigma_{\infty, E}[u]=\inf \{p\ge 0: \sup_{z\in \D} |u(z)|\rho_E^p(z)<\infty\}, \quad \inf \varnothing:=+\infty.$$
We say that a function $\psi \in BV[0,2\pi]$ belongs to $\Lambda_\gamma^*(E)$, $\gamma\in [0,1]$, if for arbitrary $\varepsilon>0$ we have $\omega_E(\delta)=O(\delta^{\gamma-\varepsilon})$, $\delta\to+$. In particular, an arbitrary $\psi \in BV[0,2\pi]$ belongs to $\Lambda_0^\gamma$, and the Lipschitz functions belong to $\Lambda_1^*$.

\begin{cor}
  Under assumptions Theorem \ref{t:appl} let $\psi=\mu([0,\theta))$, $\theta\in [0, 2\pi]$, $F=\{\theta\in [0, 2\pi]: e^{i\theta}\in E\}$. Then   $\psi \in \Lambda_\gamma^*(F)$ if and only if $\sigma_{\infty,E}[u]\le \alpha+1-\gamma$.
\end{cor}

\section{Some lemmas}

Given an arc $l$ on the unit circle, a point $w\in \D$ and a positive constant $\lambda$ we define the measure $\nu_w^\lambda(l)$ by
$$ \nu_w^\lambda(l):=\int_l \frac{|d\zeta|}{|\zeta-w|^\lambda}.$$

\begin{remark}
  For $\lambda=2$ the quantity $\nu_w^\lambda(l) (1-|w|^2) $ equals the harmonic measure of $l$. The measure $\nu_w^\lambda$ can be continued on all Borel subsets of $\partial \D$ in the standard way.
\end{remark}

The following lemma  plays an important role in our arguments.

\begin{lem} \label{l:harm_meas} Let $E^*=\bigcup_{k=1}^N [\alpha_k, \beta_k]$, where $\beta_k\le \alpha_{k+1}$, $k\in \{1, \dots, N-1\}$, $\beta_N\le \alpha_1+2\pi$, $E=\{ e^{i\theta} : \theta \in E^*\}$, $\lambda>0$, and $w\in \D$.  Then there exists a constant $C=C(\lambda)$ such that
\begin{equation}\label{e:omeg_lam_est}
  \nu_w^\lambda(E)\le C \left\{
                             \begin{array}{ll}
                               {(\rho(w))^{1-\lambda}} - {(\rho(w)+\frac{|E^*|}2)^{1-\lambda}} , & \lambda>1; \\
                               \log \Bl 1+\frac{|E^*|}{\rho(w)}\Br, & \lambda=1; \\
                               (\rho(w)+|E^*|/2) ^{1-\lambda}- (\rho(w)) ^{1-\lambda}, & \lambda<1.
                             \end{array}
                           \right.
\end{equation}
\end{lem}

\begin{proof}[Proof of Lemma \ref{l:harm_meas}]
  We write $\zeta=e^{it}$, $w=s e^{i\theta}$, $a_k=e^{i\alpha_k}$, $b_k=e^{i\beta_k}$, $k\in \{1, \dots, N\}$. Let $\rho(w)=\min _{\zeta\in E} |\zeta-w|=|\zeta_0-w|$, and $\zeta_0=e^{i\theta_0}$.  Then
\begin{equation}\label{e:omega_form}
\nu_w^\lambda(E)=\sum_{k=1}^N \int_{\alpha_k}^{\beta_k} \frac{dt}{|e^{it}-s e^{i\theta}|^\lambda}=\sum_{k=1}^N \int_{\tilde\alpha_k}^{\tilde\beta_k} \frac{dt}{|e^{it}-s|^\lambda},
\end{equation}
where $\tilde\alpha_k=\alpha_k-\theta$, $\tilde \beta_k=\beta_k-\theta$.
Dividing some segments  $[\tilde \alpha_j, \tilde\beta_j]$ onto two segments, renumerating and shifting them on  a multiple value of $2\pi$ we may achieve  that
$$ \bigcup_{k=1}^{N_0} [\tilde \alpha_k, \tilde\beta_k]\subset [0, \pi], \quad \bigcup_{k=N_0+1}^{N} [\tilde \alpha_k, \tilde\beta_k]\subset [\pi, 2\pi]$$
for some integer number $N_0$, $0\le N_0<N$, $\tilde\beta_k\le \tilde \alpha_{k+1}$, $k\in \{1, \dots, N-1\}$.

Since $\frac{1}{|e^{it}-s|^\lambda}$ is  decreasing as a function of $t$ on $[0,\pi]$ for $s\in (0,1)$, we have
\begin{equation}\label{e:lower_interval_est}
  \sum_{k=1}^{N_0}  \int_{\tilde\alpha_k}^{\tilde\beta_k} \frac{dt}{|e^{it}-s|^\lambda}\le  \int_{\tilde\alpha_1}^{\tilde\alpha_1+\sum_{k=1}^{N_0}(\tilde\beta_k-\tilde\alpha_k)} \frac{dt}{|e^{it}-s|^\lambda}.
\end{equation}
Similarly we obtain
\begin{equation}\label{e:upper_interval_est}
  \sum_{k=N_0+1}^{N}  \int_{\tilde\alpha_k}^{\tilde\beta_k} \frac{dt}{|e^{it}-s|^\lambda}\le  \int_{\tilde\beta_N}^{\tilde\beta_N-\sum_{k=N_0+1}^{N}(\tilde\beta_k-\tilde\alpha_k)} \frac{dt}{|e^{it}-s|^\lambda}.
\end{equation}

 Due to our hypothesis  $\theta_0-\theta=\tilde{\alpha}_1 (\text{\rm mod} 2\pi)$ or $\theta_0-\theta=\tilde{\beta}_N(\text{\rm mod} 2\pi)$. We may assume that
 the first equality holds. Then, taking into account that $|e^{it}-s|=|e^{i(2\pi-t)}-s|$, we obtain from \eqref{e:omega_form}--\eqref{e:upper_interval_est}
\begin{gather}\nonumber
  \nu_w^\lambda(E)\le \biggl( \int_{\tilde\alpha_1}^{\tilde\alpha_1+\sum_{k=1}^{N_0}(\tilde\beta_k-\tilde\alpha_k)} + \int_{\tilde\beta_N}^{\tilde\beta_N-\sum_{k=N_0+1}^{N}(\tilde\beta_k-\tilde\alpha_k)} \biggr)\frac{dt}{|e^{it}-s|^\lambda} \le \\
  \le 2 \int_{\tilde \alpha_0}^{\tilde\alpha_0 +\frac 12 \sum_{k=1}^N (\beta_k-\alpha_k)} \frac{dt}{|e^{it}-s|^\lambda}\label{e:omega_main_est}
\end{gather}
We denote $\delta=\frac 12 |E^*|$. Using the elementary relationship $x^2+y^2\le (x+y)^2\le 2(x^2+y^2)$, $x,y\ge 0$ and standard estimates we deduce for $\frac12\le s<1$
\begin{gather}\nonumber
\int_{\tilde \alpha_0}^{\tilde \alpha_0+\delta} \frac{dt}{|e^{it}-s|^\lambda}=\int_{\tilde\alpha_0}^{\tilde \alpha_0+\delta }\frac{dt}{((4s\sin^2 t/2+(1-s)^2)^{\frac{\lambda}{2}}}\le \\ \le \int_{\tilde\alpha_0}^{\tilde \alpha_0+\delta}\frac{dt}{(\frac{2}{\pi^2}t^2+(1-s)^2)^{\frac{\lambda}{2}}}  \le
\int_{\tilde\alpha_0}^{\tilde \alpha_0+\delta }\frac{2^{\lambda/2}dt}{(\frac{\sqrt{2}}{\pi} t+(1-s))^{\lambda}} .\label{e:nu_la_gen}
\end{gather}
The latter integral can be computed explicitly depending on $\lambda$. We start with the case $\lambda>1$.
Then \eqref{e:omega_main_est}, \eqref{e:nu_la_gen} imply 
\begin{gather}\nonumber
   \nu_w^\lambda(E)\le 2\int_{\tilde\alpha_0}^{\tilde \alpha_0+\delta} \frac{dt}{|e^{it}-s|^\lambda} \le \\ \le \nonumber  C(\lambda)\Bigl( \frac{1}{(\tilde \alpha_0+1-s)^{\lambda-1}} -\frac{1}{(\tilde \alpha_0+\delta+1-s)^{\lambda-1}}\Bigr) = \\ = \nonumber
 \frac{C(\lambda)}{(\tilde \alpha_0+1-s)^{\lambda-1}} \biggl( 1- \frac{1}{\Bl 1+ \frac{\delta}{\tilde \alpha_0+1-s} \Br^{\lambda-1}} \biggr) \le \\ \le
 \frac{C(\lambda)}{\rho(w)^{\lambda-1}} \biggl( 1- \frac{1}{\Bl 1+ \frac{|E^*|}{2\rho(w)}\Br^{\lambda-1}} \biggr)=
\frac{C(\lambda)}{\rho(w)^{\lambda-1}}- \frac{C(\lambda)}{(\rho(w)+ \frac{|E^*|}{2})^{\lambda-1}}  \label{e:omega_la_final}
\end{gather}
as required.

Let $\lambda\in (0,1)$. Similarly, we deduce
\begin{gather}\nonumber
   \nu_w^\lambda(E)\le  C(\lambda)\Bigl((\tilde \alpha_0+\delta+1-s)^{1-\lambda} - (\tilde \alpha_0+1-s)^{1-\lambda}\Bigr) = \\ = \nonumber
C(\lambda) {(\tilde \alpha_0+1-s)^{1-\lambda}} \biggl( {\Bl 1+ \frac{\delta}{\tilde \alpha_0+1-s} \Br^{1-\lambda}}-1  \biggr) \le \\ \le \nonumber
 C(\lambda) {\rho(w)^{1-\lambda}} \biggl( {\Bl 1+ \frac{|E^*|}{2\rho(w)}\Br^{1-\lambda}}-1 \biggr)=\\ =C(\lambda)((\rho(w)+|E^*|/2) ^{1-\lambda}- (\rho(w)) ^{1-\lambda}) \label{e:omega_la_final'}
\end{gather}
Finally, for $\lambda=1$ we get
$$ \nu_w^1(E)\lesssim \log \Bl \frac{\tilde{\alpha}_0+\delta+1-s}{\tilde{\alpha}_0+1-s} \Br \lesssim \log \Bl 1+ \frac{|E^*|}{2\rho(w)}\Br .$$
The assertion of the theorem follows from \eqref{e:omega_la_final}, \eqref{e:omega_la_final'} and the latter estimate.
\end{proof}

\section{Proofs of the theorems}

\begin{proof}[Proof of Theorem \ref{t:harm_conj}]
  We start we the case $\gamma< q$.  We write $F(z):=u(z)+i\tilde u(z)$, so $F$ is analytic in $\D$. Then
\begin{gather}\nonumber
  F(z)=\frac{1}{2\pi} \int_{0}^{2\pi} \frac{Re^{i\theta}+z}{Re^{i\theta}-z} u(Re^{i\theta})\, d\theta +i \Im F(0)=\\ = \frac{1}{2\pi} \int_{0}^{2\pi} \frac{1+ \frac zR  e^{-i\theta}}{1-\frac zR e^{-i\theta}} u(Re^{i\theta})\, d\theta +i \Im F(0).\label{e:F_repres}
  \end{gather}
We are going to apply Djrbashian the operator of fractional derivation  $D^\alpha (r^\alpha F(re^{i\varphi}))$ where $D^\alpha$ is the Riemann-Liouville operator of order  $\alpha  >0$. It is known (\cite[p.577]{Dj}) that $D^\alpha(r^\alpha S_0(re^{i\varphi}))=S_\alpha(re^{i\varphi})$, where $S_{\al}(z)=\Gamma(1+\alpha)\Bigl(
 \frac 2{(1-z)^{\alpha+1}} -1\Bigr)$.
%P_q(z) =\Re S_q(z), \quad S_q(0)=\Gamma(q+1).$$
Note that $S_0$ is the Schwarz  kernel. The same arguments show that $D^\alpha(r^\alpha S_0(\frac rr e^{i\varphi}))=S_\alpha(\frac rR e^{i\varphi})$.

Taking in account the  equality $D^\alpha(r^\alpha)=\Gamma (1+\alpha)$, $\alpha>0$,  we get
\begin{gather*}F_\alpha(re^{i\varphi}):= D^\alpha (r^\alpha F(re^{i\varphi})) =\\ =\frac1{2\pi}\int_0^{2\pi} S_\alpha\Bigl(\frac  zR e^{-i\theta}\Bigr) u(Re^{i\theta})\, d\theta + i \Im F(0) \Gamma(1+\alpha).\end{gather*}
Then
\begin{gather}\nonumber
 | F_\alpha(z)|\le \frac{C_0\Gamma(\alpha+1)}{2\pi}\int_0^{2\pi} \Bl  \frac 2{|1- \frac zR e^{-i\theta}|^{\alpha+1}}+ 1 \Br \frac{(1-R)^\gamma}{\rho^q(Re^{i\theta})} \, d\theta +\\ + |\Im F(0)|\Gamma(\alpha+1)\le  \nonumber  \\ \le
  C (1-R)^\gamma \int_0^{2\pi} \frac { d\theta}{|1- \frac zR e^{-i\theta}|^{\alpha+1}\rho^q(Re^{i\theta})} +C.\label{e:F_al_est}
\end{gather}

%Given an arc $l$ on the unit circle, a point $w\in \D$ and a positive constant $\lambda$ we define the measure $\nu_w^\lambda(l)$ by
%$$ \nu_w^\lambda(l):=\int_l \frac{|d\zeta|}{|\zeta-w|^\lambda}.$$
%
%\begin{remark}
%  For $\lambda=2$ the quantity $\nu_w^\lambda(l) (1-|w|^2) $ equals the harmonic measure of $l$. The measure $\nu_w^\lambda$ can be continued on all Borel subsets of $\partial \D$ in the standard way.
%\end{remark}

We write $$ \nu_w^\lambda(\theta):=\nu_w^\lambda([0, \theta))=\int_0^\theta \frac{dt}{|e^{it}-w|^\lambda}.$$
Due to `layer cake representation' (\cite[Theorem 1.13]{LeLo}) we have
\begin{gather}\nonumber
  \int_{0}^{2\pi}\frac { d\theta}{|1- \frac zR e^{-i\theta}|^{\alpha+1}\rho^q(Re^{i\theta})}=\int_0^{2\pi} \frac{d\nu_{z/R}^{\alpha+1}(\theta)}{\rho^q(Re^{i\theta})}= \\ = \nonumber
  q\int_{0}^{\infty} y^{q-1} \nu_{z/R}^{\alpha+1}\Bl \{ \theta: \frac{1}{\rho(Re^{i\theta})}>y\} \Br \, dy=\\ = \nonumber
  q\int_{0}^{\frac1{1-R}} y^{q-1} \nu_{z/R}^{\alpha+1}\Bl \Bigl\{ \theta: \rho(Re^{i\theta})<\frac 1y\Bigr\} \Br \, dy\le   \\
\le C(q, \lambda, \alpha) \int_{0}^{\frac1{1-R}} \frac{y^{q-1}}{\rho^\alpha(z/R)}\, dy \le \frac{C}{(1-R)^q\rho^\alpha(z/R)}, \quad  \frac R2\le |z|<R.
\label{e:layer}
\end{gather}
Substituting this estimate into \eqref{e:F_al_est} we get
$$ |F_\alpha(z)|\le \frac{C}{(1-R)^{q-\gamma}\rho^\alpha(z/R)}, \quad  \frac R2\le |z|<R.$$

We choose $R=(1+r)/2$, $|z|=r$, and suppose that $r\in [1/2, 1)$.  We then  have
\begin{gather} \nonumber\allowdisplaybreaks
  |F(re^{i\varphi})|=|r^{-\alpha} D^{-\alpha} F_\alpha(re^{i\varphi})|=\frac{r^{-\alpha}}{\Gamma(\alpha)}  \Bigl| \int_0^r (r-t)^{\alpha-1} F_\alpha(te^{i\varphi})\, dt\Bigr| \le  \\ \le
C  \int_0^r \frac{(r-t)^{\alpha-1}\, dt}{\Bl\frac{1-t}{2}\Br^{q-\gamma}\rho^\alpha \bigl( \frac{2t e^{i\varphi}}{1+t}\bigr)}\le  \frac{C}{\rho^{\alpha}\bigl( \frac{2r e^{i\varphi}}{1+r}\bigr)} \int_0^r \frac{(r-t)^{\alpha-1}}{(1-t)^{q-\gamma}} \, dt.
  \label{e:|F|est}
\end{gather}
We have used the estimate $\rho(z)\le 2\rho (\tau z)$, $\tau \in (0,1)$, $z\in \D$ (\cite[p.41]{FG09}).
%Standard estimates give \int_{r-(1-r)}^r \frac{(r-t)}{}

We now consider two subcases. First, let $q>\gamma\ge 0$. We then choose $\alpha=q-\gamma>0$.  It follows from \eqref{e:|F|est} that
$$ |F(re^{i\varphi})|\le \frac{C}{\rho^{q-\gamma}\bigl( \frac{2r e^{i\varphi}}{1+r}\bigr)} \int_0^r \frac{(r-t)^{q-\gamma-1}}{(1-t)^{q-\gamma}} \, dt\le  \frac{C \log \frac 1{1-r}}{\rho^{q-\gamma}\bigl( \frac{2r e^{i\varphi}}{1+r}\bigr)}.$$

Next, let $q>0>\gamma$. Then we take $\alpha=q$. The estimate \eqref{e:|F|est} yields
$$ |F(re^{i\varphi})|\le \frac{C}{\rho^{q}\bigl( \frac{2r e^{i\varphi}}{1+r}\bigr)} \int_0^r \frac{(r-t)^{q-1}}{(1-t)^{q-\gamma}} \, dt\le  \frac{C (1-r)^{\gamma}}{\rho^{q}\bigl( \frac{2r e^{i\varphi}}{1+r}\bigr)}.$$

Finally, we consider the case $q=\gamma>0$.  We then estimate $F(z)$ using \eqref{e:F_repres}.
\begin{gather*}
|F(z)|\lesssim  \int_0^{2\pi} \Bigl( \frac{1}{|1-\frac zR e^{-i\theta}|} +1  \Bigr) |u(Re^{i\theta})|\, d\theta +1 \lesssim  \\ \lesssim
  \int_0^{2\pi}  \frac{(1-R)^q}{|1-\frac zR e^{-i\theta}|\rho^q(R e^{i\theta})}\, d\theta +1 \lesssim  \\
\lesssim      (1-R)^q     \int_0^{2\pi} \frac{d\nu_{z/R}^{1}(\theta)}{\rho^q(Re^{i\theta})}+1 = \\ =
  (1-R)^q\int_{0}^{\infty} y^{q-1} \nu_{z/R}^{1}\Bl \{ \theta: \frac{1}{\rho(Re^{i\theta})}>y\} \Br \, dy +1 \lesssim \\
\lesssim  (1-R)^q \int_{0}^{\frac1{1-R}} y^{q-1} \log \Bl 1+\frac{2\pi}{\rho(z/R)} \Br \, dy+1 \lesssim   \\
\lesssim    \log \frac{2\pi+2}{\rho(z/R)}, \quad  \frac R2\le |z|<R.
                                                                                               \end{gather*}
To finish the proof, we observe that $$\rho\Bl \frac{1+r}{2}e^{i\varphi}\Br -\rho\Bl re^{i\varphi}\Br \ge -\frac{1-r}{2}\ge -\rho\Bl \frac{1+r}{2}e^{i\varphi}\Br,\quad 0<r<1, $$
so $\rho\Bl \frac{1+r}{2}e^{i\varphi}\Br \ge \frac 12 \rho(re^{i\varphi})$.
\end{proof}
\begin{proof}[Proof of Theorem \ref{t:appl}]
%irst, we consider the case $\gamma=1$. Note that the class
%$\Lambda_1$ consists of functions that are integrals of bounded
%functions. Thus it is sufficient to apply Theorem (6.3)
%\cite[Ch.IV]{Z}, which states that (\ref{e:hl})
% holds if and only if $B(r,u)$ is bounded as $r \uparrow  1$.
%
%Consider the case $\gamma\in (0,1)$.

 The {\it proof of (i)} is standard
(cf.\  \cite{Sh2}).

The following estimates of $P_\alpha(r,t)$ are straightforward
\begin{equation}\label{e:pe}
  \Bigl| \frac \partial{\partial t} P_\alpha(re^{it})\Bigr| \lesssim  \min \biggl\{
  \frac{1}{(1-r)^{\alpha+2}},
  \frac{1} {t^{\alpha+2}}\biggr\} , \; r\ge \frac 12, |t|\le
  \pi.
\end{equation}
%We extend $\psi$ on $\R$ by the formula $\psi(t+2\pi)-\psi(t)=
%\psi(2\pi)-\psi(0)$.
%Since $ P_0(r,t)$ is a periodic and even function in $t$,
Let $\psi(t):=\mu([0, t))$, $t\in [-\pi, \pi]$. We have (cf. \cite{Ch1})
\begin{gather*}
u(re^{i\varphi})= \int\limits_{-\pi+\vfi}^{\pi+\vfi}
P_\alpha(re^{i(\theta-\vfi)}) d(\psi(\theta) - \psi(\varphi)) =\\
%={(\psi(\theta)-\psi(\varphi))P_0(r, \theta-\varphi)}
%\Bigr|_{-\pi+\varphi}^{\pi+\varphi}- \\ -
%\int\limits_{-\pi+\vfi}^{\pi+\vfi} \frac \partial{\partial \theta}
%(P_0(r, \theta -\varphi)) (\psi(\theta) - \psi(\varphi))d\theta
 ={(\psi(2\pi)-\psi(0))P_\alpha(-r)} - \int\limits_{-\pi}^{\pi}
\frac \partial{\partial \tau} (P_\alpha(re^{i\tau})) (\psi(\tau+\varphi) -
\psi(\varphi))d\tau.
\end{gather*}
Hence, using (\ref{e:pe}), we obtain
\begin{gather}\nonumber
|u(re^{i\varphi})| \le O(1)+ \int_{E^*}\frac {|\psi(t) -
\psi(\varphi)|}{|re^{i\varphi}-e^{it}|^{2+\alpha}} dt\le \\ \nonumber
\lesssim 1 + \int\limits_{|\varphi -t|\le \rho(re^{i\varphi})}
\frac{\omega(|t-\varphi|,\psi)}{\rho(re^{i\varphi})^{2+\alpha}} \, dt +
 \int\limits_{\rho(re^{i\varphi})<|t-\varphi|\le \pi} \frac
{ \omega(|t-\varphi|,\psi)}{|t-\varphi|^{2+\alpha}} \, dt \lesssim
\nonumber \\ \lesssim
1 + \frac{1}{\rho(re^{i\varphi})^{2+\alpha}} \int\limits_{0}^{\rho(re^{i\varphi})}
 \tau^\gamma \, d\tau +
 \int\limits_{\rho(re^{i\varphi})}^\pi \frac 1{\tau ^{2+\alpha-\gamma}}
 \, d\tau\lesssim (\rho(re^{i\varphi}))^{-\alpha+ \gamma-1}, \quad r\uparrow 1 . \label{e:mce}
\end{gather}
%Here, we have used increasing of the modulus of continuity.
%Since $\psi\in \Lambda_\gamma$, $\omega(\tau, \psi; \varphi)=
%O(\tau^\gamma)$ as $\tau\downarrow 0$. Thus,
%(\ref{e:mce}) yields
%$$ B(r,u)\le C_2(\gamma) (1-r)^{\gamma -1}, \quad r\uparrow1.$$

{\it Proof of (ii).} % Let $u(re^{i\varphi})$ be harmonic for $r<1$,
%and $\int_{0}^{2\pi} |u(re^{i\varphi})|\, d\varphi\le  C_3$.
We start with the case $\alpha=0$.

\smallskip
\noi{\bf Remark 5.} By Nevanlinna's Theorem, %we have (\ref{e:rep1}), where
%$\psi\in BV [0,2\pi]$, and one can take
 at any
point $\theta$ of continuity of $\psi$ for some sequence $(r_n)$
(\cite{Dj},~\cite[p.57]{Pr}).
\begin{equation}\psi(\theta)=\lim_{r_n\uparrow 1} \int_0^{\theta} u(r_n e^{i\phi})\, d\phi.
\label{e:priv} \end{equation}
%Actually, $\psi$ in representation (\ref{e:rep1}) is unique up to
%a constant summand,  moreover
%\psi(\theta)=\lim_{r\uparrow 1} \int_0^{\theta} u(r e^{i\phi})\, d\phi

\smallskip
Let $F(z)=u(z)+iv(z)$ be  an analytic  function in  $\D$.
By Theorem \ref{t:harm_conj}
$|v(re^{i\varphi})|=O((\rho(re^{i\varphi}))^{\gamma-1}\log \frac{1}{1-r})$
as $r\uparrow 1$.

Define the analytic function
$\Phi(z)=\int_0^z F(\zeta)\, d\zeta$, $z\in D$.
For  any fixed $\varphi\in [0,2\pi]$ and  $0<r'< r''< 1$, we have
%\begin{equation*}
%  |\Phi(re^{i\varphi})|= \Bigl| \int_0^r F(\rho e^{i\varphi})
%  e^{i\varphi}\, d\rho\Bigr| \le C_3 \int_0^r
%  (1-\rho)^{\gamma-1}\, d\rho=C_4(\gamma), \quad r\uparrow 1.
%\end{equation*}
%Moreover,
 \begin{gather*}
  |\Phi(r''e^{i\varphi})-\Phi(r'e^{i\varphi})|= \Bigl| \int_{r'}^{r''} F(\rho e^{i\varphi})
  e^{i\varphi}\, d\rho\Bigr| \le \\ \le C_4 \int_{r'}^{r''}
  (1-\rho)^{\gamma-1}\log \frac 1{1-\rho}\, d\rho\le \frac{C_4}{\gamma} (1-r')^{\gamma}\log \frac 1{1-r'} .
\end{gather*}
Therefore, by Cauchy's criterion, there exists
$\lim_{r\uparrow1} \Phi(re^{i\varphi})\equiv \Phi(e^{i\varphi})$
uni\-form\-ly in $\varphi$. Consequently, $\tilde{\Phi}(\varphi)\stackrel{\rm def}=\Phi(e^{i\varphi})$ is a
continuous function on $[0, 2\pi]$.

%Let us prove that $\tilde{\Phi} \in \Lambda_\gamma$.
Let $h\in (0,1)$, $z_0=e^{i\varphi}$, $z_1=(1-h)e^{i\varphi}$,
$z_2=(1-h)e^{i(\varphi+h)}$, $z_3=e^{i(\varphi+h)}$, and $e^{i\varphi}\in E$.

Then  by Cauchy's theorem
\begin{gather*}
  \Phi(z_3)-\Phi(z_0)=\int_{[0,z_3]} F(z)\, dz +
  \int_{[z_0, 0]} F(z)\, dz= \\
=  \biggl(\int_{[z_0,z_1]} + \int_{z_1} ^{z_2} + \int_{[z_2,z_3]}
\biggr) F(z)\, dz.
\end{gather*}
For sufficiently small $h>0$, we have
\begin{gather}\nonumber
 \Bigl|\int_{[z_0,z_1]} F(z)\, dz\Bigr|\lesssim \int_{1-h}^1
\frac{\log \frac 1{1-r}}{(\rho(re^{i\varphi}))^{1-\gamma}} dr \le \\ \le
\int_{1-h}^1
\frac{\log \frac 1{1-r}}{(1-r)^{1-\gamma}} dr \lesssim  h^\gamma \log \frac{1}{h}, \quad h\to0+. \label{e:segment_est}
\end{gather}

Similarly, $\Bigl|\int_{[z_2,z_3]} F(z)\, dz\Bigr|\lesssim  h^\gamma \log \frac{1}{h}.$
It is obvious that $$|\int_{z_1} ^{z_2}  F(z)\, dz| \lesssim \log \frac{1}{h} \int_{\varphi}^{\varphi+h} \frac{d\theta}{\rho^{1-\gamma}((1-h)e^{i\theta})}
 \lesssim  h^\gamma \log \frac{1}{h}.$$
If $h<0$ the similar estimates hold.
Therefore, \begin{equation}\label{e:Phi_smooth}
             \biggl|\Phi(e^{i(\varphi+h)})- \Phi(e^{i\varphi})\biggr|\lesssim
h^\gamma\log \frac{1}{h}, \quad h\to0+.
           \end{equation}

For $R\in (0,1)$ define
\begin{gather*}
\lambda_R(\theta)=\int_0^\theta  F(Re^{i\sigma})\, d\sigma=
\int_0^\theta \frac{d\Phi(Re^{i\sigma})}{iRe^{i\sigma}}=\\
=\frac{\Phi(Re^{i\theta})}{iRe^{i\theta}} -\frac{\Phi(R)}{iR} +
\frac 1R \int_0^\theta \Phi(Re^{i\sigma}) e^{-i\sigma}\, d\sigma.
\end{gather*}
Since $\Phi(z)$ is continuous in $\{z :|z|\le 1\}$,
$\Phi(Re^{i\sigma}) \text{\raisebox{-5pt}{$\rightrightarrows  \atop
 \theta$}}
\Phi(e^{i\sigma})$
as $R\uparrow 1$.
Consequently, as $R\uparrow 1$
\begin{equation*}
\lambda_R(\theta) \rightrightarrows-\Phi(e^{i\theta})ie^{-i\theta}
+i\Phi(1) +\int_0^\theta \Phi(e^{i\sigma}) e^{-i\sigma}\, d\sigma
\equiv \lambda(\theta) \in C[0,2\pi].
  \end{equation*}
Combining the latter relationship with \eqref{e:Phi_smooth}, we deduce
\begin{equation}\label{e:la_smooth}
             |\lambda(e^{i(\varphi+h)})- \lambda (e^{i\varphi})|\lesssim
h^\gamma\log \frac{1}{h}, \quad h\to0+.\end{equation}
 On the other hand,
$%\begin{equation*} %\label{e:rep1}
u(re^{i\vfi})= \int_0^{2\pi} P_0(r, \vfi-t) d\psi(t)$,
where, by (\ref{e:priv}) and the definition of $\lambda$
$$\psi(\theta)=\lim_{r\uparrow 1} \int_0^{\theta} \Re F(r e^{i\phi})\,
d\phi=\lambda(\theta).
$$
Thus, $\psi$ satisfies required smoothness properties in the case $\alpha =0$.

Let now  $\alpha>0$. We need the following lemma.
\begin{alemma} [{\cite[Lemma 14]{Illin}}] \label{l:gaint} Let $0\le \ga < \al <\infty$. Then there exists a constant $C(\ga, \al)>0$ such that
\begin{equation}\label{e:gaint}
D^{-\ga} \frac{1}{|1-r\zeta|^\al} \le \frac{C(\ga, \al
)}{|1-r\zeta|^{\al-\ga}}, \quad \zeta \in \overline{\D}, 0<r<1.
\end{equation}
\end{alemma}

Applying this lemma we obtain, that $u_\alpha(z)=\int_{\partial D}  P_0(z\bar\zeta) d\mu(\zeta)$ satisfies
 \begin{equation}\label{e:u_al_growth}
    |u_\alpha(z)|\le \frac{C}{d^{1-\gamma}(z)}, \quad z\in \D.  \end{equation}
It remains  to apply proved assertion of (i) for $\alpha=0$.
 \end{proof}
%\newpage

\bigskip
{{\it Address 1:} School of Mathematics Science,
           Guizhou  Normal University, Guiyang, Guizhou 550001, China
%Faculty of Mathematics and Computer Sciences, University of Warmia and Mazury in Olsztyn,
%S\l oneczna 54, 10-710, Osztyn, Poland
{\it e-mail:} chyzhykov@yahoo.com}

{\it Address 2:} Faculty of Mechanics and Mathematics, Lviv Ivan Franko National University,
Universytets'ka 1, 79000,
 Lviv,  Ukraine

\end{document}